\documentclass{article}
\usepackage{amsmath,amssymb,amsthm,fullpage}
\usepackage{graphicx,color,epsfig,epstopdf}
\usepackage{authblk}
\usepackage{enumerate}

\newcommand{\DEF}{\sl}
\newcommand{\STAB}{\mathop{\mathrm{STAB}}}

\newcommand{\conv}{\mathop{\mathrm{conv}}}
\newcommand{\rk}{\mathop{\mathrm{rank}}} 
\newcommand{\nnegrk}{\mathop{\mathrm{rank}_+}} 
\newcommand{\psdrk}{\mathop{\mathrm{rank}}\nolimits_{\mathrm{PSD}}} 
\newcommand{\supp}{\mathop{\mathrm{supp}}} 
\newcommand{\suppmat}{\mathop{\mathrm{suppmat}}} 
\newcommand{\vertexset}{\mathop{\mathrm{vert}}}
\newcommand{\RR}{\mathbb{R}}

\newcommand{\lfactor}{T}
\newcommand{\rfactor}{U}

\newcommand{\Tr}{\mathop{\mathrm{Tr}}}
\newcommand {\tr}[1]{\Tr\left[#1\right]}
\newcommand {\mbS}{\mathbb{S}}

\newcommand{\diag}{\mathop{\mathrm{diag}}}
\newcommand{\xc}{\mathop{\mathrm{xc}}}
\newcommand{\xcp}{\mathop{\mathrm{xc}}\nolimits_{\mathrm{SDP}}}
\newcommand{\TSP}{\mathop{\mathrm{TSP}}}
\newcommand{\CUT}{\mathop{\mathrm{CUT}}}
\newcommand{\COR}{\mathop{\mathrm{COR}}}

\newcommand{\ketbra}[2]{|#1\rangle\langle#2|}

\newcommand{\ket}[1]{|#1 \rangle}
\newcommand{\bra}[1]{\langle #1|}
\def\01{\{0,1\}}

\newtheorem{thm}{Theorem}
\newtheorem{lem}[thm]{Lemma}

\newtheorem{cor}[thm]{Corollary}

\newtheorem{remark}{Remark}

\title{Exponential Lower Bounds for Polytopes in Combinatorial Optimization}

\author[1]{Samuel Fiorini\thanks{Partially supported by the \emph{Actions de Recherche Concert\'ees} (ARC) fund of the French community of Belgium.}}
\author[3]{Serge Massar\thanks{Partially supported 
by the European Commission under the project QCS (Grant No.\ 255961).}}
\author[2]{Sebastian Pokutta}
\author[1]{Hans Raj Tiwary\thanks{Postdoctoral Researcher of the \emph{Fonds National de la Recherche Scientifique} (F.R.S.--FNRS).}}
\author[4]{Ronald de~Wolf\thanks{Partially supported by a Vidi grant from the Netherlands Organization for Scientific Research (NWO), and by the European Commission under the project QCS (Grant No.\ 255961).}}

\affil[1]{\small Department of Mathematics, Universit\'e libre de Bruxelles CP 216, Boulevard du Triomphe, 1050
Brussels, Belgium. \textit{Email: }\{sfiorini, htiwary\}@ulb.ac.be}
\affil[2]{\small Industrial and Systems Engineering (ISyE)
and Algorithms and Randomness Center (ARC), Georgia Institute of Technology, Groseclose 0205, Atlanta, GA 30332, USA. \textit{Email: } sebastian.pokutta@isye.gatech.edu}
\affil[3]{\small Laboratoire d'Information Quantique, Universit\'e libre de Bruxelles CP 225, Boulevard du Triomphe, 1050
Brussels, Belgium. \textit{Email: }smassar@ulb.ac.be}
\affil[4]{\small CWI and University of Amsterdam, Amsterdam, The Netherlands. \textit{Email: }rdewolf@cwi.nl}

\begin{document}

\maketitle

\begin{abstract}
We solve a 20-year old problem posed by Yannakakis and prove that there exists no polynomial-size linear program (LP) whose associated polytope projects to the traveling salesman polytope, even if the LP is not required to be symmetric. Moreover, we prove that this holds also for the cut polytope and the stable set polytope. These results were discovered through a new connection that we make between one-way quantum communication protocols and semidefinite programming reformulations of LPs.
\end{abstract}


\section{Introduction}

Since the advent of the simplex method~\cite{Dantzig1947}, linear programming has become a prominent tool for solving optimization problems in practice. On the theoretical side, LPs can be solved in polynomial time via either the ellipsoid method~\cite{Khachiyan1979} or interior point methods~\cite{Karmakar1984}. 

In 1986--1987 there were attempts~\cite{Swart86} to prove P $=$ NP by giving a polynomial-size LP that would solve the traveling salesman problem (TSP). Due to the large size and complicated structure of the proposed LP for the TSP, it was difficult to show directly that the LP was erroneous. In a groundbreaking effort to refute all such attempts, Yannakakis \cite{Yannakakis1988} proved that every symmetric LP for the TSP has exponential size (see~\cite{Yannakakis91} for the journal version). Here, an LP is called \emph{symmetric} if every permutation of the cities can be extended to a permutation of all the variables of the LP that preserves the constraints of the LP. Because the proposed LP for the TSP was symmetric, it could not possibly be correct.

In his paper, Yannakakis left as a main open problem the question of proving that the TSP admits no polynomial-size LP, \emph{symmetric or not}. We solve this question by proving a super-polynomial lower bound on the number of inequalities in \emph{every} LP for the TSP. We also prove such unconditional super-polynomial lower bounds for the maximum cut and maximum stable set problems. Therefore, it is impossible to prove P = NP by means of a polynomial-size LP that expresses any of these problems. Our approach is inspired by a close connection between semidefinite programming reformulations of LPs and one-way quantum communication protocols that we introduce here.

\subsection{State of the Art}

%

\paragraph{From Problems to Polytopes}

For combinatorial optimization problems such as the TSP, the feasible solutions can be encoded as points in a set $X \subseteq \{0,1\}^d$ in such a way that solving an instance of the problem amounts to optimizing a linear objective function over $X$, with coefficients given by the instance. By taking the convex hull of $X$, one obtains a polytope $P := \conv{X}$ (see Appendix~\ref{apx:background} for background on polytopes). Optimizing any linear function $f(x)$ over $X$ is equivalent to optimizing this function $f(x)$ over $P = \conv{X}$. 

For example, for the TSP we have a set $X \subseteq \{0,1\}^{n \choose 2}$ of 0/1-points that correspond to a Hamiltonian cycle in the complete $n$-vertex graph $K_n$. The convex hull of these points is the \emph{TSP polytope} $\TSP(n) = \conv{X}$. An instance of the TSP is given by the set of edge-weights $w_{ij}$. Solving this instance amounts to minimizing $f(x) := \sum_{i<j}w_{ij}x_{ij}$ over all $x \in \TSP(n)$.  This minimum is attained at a vertex of the polytope, i.e., at a point $x \in X$.

The idea of representing the set of feasible solutions of a problem by a polytope forms the basis of a standard and powerful methodology in combinatorial optimization, see, e.g., \cite{SchrijverBook}. 

\paragraph{Extended Formulations and Extensions}

Resuming the discussion above (and assuming that the problem is a minimization problem), we have $\min \{f(x) \mid x \in X\} = \min \{f(x) \mid x \in P\} = \min \{f(x) \mid Ax \leqslant b\}$, where $Ax \leqslant b$ is any linear description of $P$. This turns any given instance of the combinatorial optimization problem into an LP, however over an implicit system of constraints the LP is potentially large since it has at least one inequality per facet of $P$. In fact, even for polynomially solvable problems, the associated polytope $P$ may have an exponential number of facets. 

By working in an extended space, that is, considering extra variables $y \in \RR^k$ besides the original variables $x \in \RR^d$, it is often possible to decrease the number of constraints. In some cases, a polynomial increase in dimension can be traded for an exponential decrease in the number of constraints. This is the idea underlying extended formulations. 

Formally, an \emph{extended formulation} (EF) of a polytope $P \subseteq \RR^d$ is a linear system
\begin{equation} \label{eq:EF}
E^{=}x + F^{=}y = g^{=}, \ E^{\leqslant}x + F^{\leqslant}y \leqslant g^{\leqslant}
\end{equation}
in variables $(x,y) \in \RR^{d+r},$ where $E^{=}, F^{=}, E^{\leqslant}, F^{\leqslant}$ are real matrices with $d, k, d, k$ columns respectively, and $g^{=}, g^{\leqslant}$ are column vectors, such that $x \in P$ if and only if there exists $y$ such that \eqref{eq:EF} holds. 

The \emph{size} of an EF is defined as the number of \emph{inequalities} in the system. Another possible definition of size would be the sum of the number of variables and total number of constraints (equalities plus inequalities) defining the EF. This would make little difference because if a polytope $P \subseteq \mathbb{R}^d$ has an EF with $r$ inequalities, then it has an EF with $d+r$ variables, $r$ inequalities and at most $d+r$ equalities (see Remark \ref{rem:factorization} for a proof). If we assume that $P$ is full-dimensional (otherwise one may cheat and make $d$ artificially high) 
then $d \leqslant r$ and thus the two measures of size are within a constant of each other.

Notice that optimizing any (not necessarily linear) objective function $f(x)$ over all $x \in P$ amounts to optimizing $f(x)$ over all $(x,y) \in \RR^{d+r}$ satisfying \eqref{eq:EF}, provided \eqref{eq:EF} defines an EF of $P$.

Here, we often restrict to EFs in slack form, that is, containing only equalities and one nonnegativity inequality per additional variable:
\begin{equation} \label{eq:EF-slack-form}
E x + F y = g,\ y \geqslant \mathbf{0}.
\end{equation}
The proof of the factorization theorem (Theorem~\ref{thm:factorization}) shows that this can be done without loss of generality, see Remark~\ref{rem:factorization}. In the following we put EFs in slack form to ease the generalization to arbitrary cones. Notice that the size of an EF in slack form can equivalently be defined as the number of additional variables since the only inequalities are from $y \geqslant \mathbf{0}$.
 
An \emph{extension} of the polytope $P$ is another polytope 
$Q \subseteq \mathbb{R}^e$ such that $P$ is the image of $Q$ under a linear map. We define the \emph{size} of an extension $Q$ as the number of facets of $Q$. If $P$ has an extension of size $r$, then it has an EF of size $r$. Conversely, it is known that if $P$ has an EF of size $r$, then it has an extension of size at most~$r$ (see Theorem~\ref{thm:factorization}). In this sense, the concepts of EF and extension are equivalent.

\paragraph{The Impact of Extended Formulations}

EFs have pervaded the areas of discrete optimization and approximation algorithms for a long time. For instance, Balas's disjunctive programming~\cite{Balas1985}, the Sherali-Adams hierarchy~\cite{SheraliAdams1990}, the Lov\'asz-Schrijver closures~\cite{LovaszSchrijver1991}, lift-and-project~\cite{BalasCeriaCornuejols1993}, and configuration LPs are all based on the idea of working in an extended space. Recent surveys on EFs in the context of combinatorial optimization and integer programming are 
\cite{ConfortiCornuejolsZambelli10,VanderbeckWolsey2010,Kaibel11,Wolsey11}.

\paragraph{Symmetry Matters}

Yannakakis \cite{Yannakakis91} proved a $2^{\Omega(n)}$ lower bound on the size of any \emph{symmetric} EF of the TSP polytope $\TSP(n)$ (defined above and in Section~\ref{sec:TSP_polytopes}). Although he remarked that he did ``not think that asymmetry helps much'', it was recently shown by Kaibel et al. \cite{KaibelPashkovichTheis10} (see also \cite{Pashkovich09}) that symmetry is a restriction in the sense that there exist polytopes that have polynomial-size EFs but no polynomial-size symmetric EF. This revived Yannakakis's tantalizing question about unconditional lower bounds. That is, bounds which apply to the \emph{extension complexity} of a polytope $P$, defined as the minimum size of an EF of $P$ (irrespective of any symmetry assumption).

\paragraph{0/1-Polytopes with Large Extension Complexity}

The strongest unconditional lower bounds so far were obtained by Rothvo\ss \cite{Rothvoss11}. By an elegant counting argument inspired by Shannon's theorem~\cite{Shannon49}, it was proved that there exist $0/1$-polytopes in $\mathbb{R}^d$ whose extension complexity is at least $2^{d/2-o(d)}$. However, Rothvo\ss's counting technique does not provide \emph{explicit} 0/1-polytopes with an exponential extension complexity.

\paragraph{The Factorization Theorem}

Yannakakis \cite{Yannakakis91} discovered that the extension complexity of a polytope $P$ is determined by certain factorizations of an associated matrix, called the \emph{slack matrix} of $P$, that records for each pair $(F,v)$ of a facet $F$ and vertex $v$, the algebraic distance of $v$ to a valid hyperplane supporting $F$. Defining the \emph{nonnegative rank} of a matrix $M$ as the smallest natural number $r$ such that $M$ can be expressed as $M = TU$ where $T$ and $U$ are nonnegative matrices (i.e., matrices whose elements are all nonnegative) with $r$ columns (in case of \(T\)) and $r$ rows (in case of \(U\)), respectively, it turns out that the extension complexity of every polytope $P$ is exactly the nonnegative rank of its slack matrix.

We point out that this result generalizes to \emph{any} slack matrix of the polytope, which may contain additional rows corresponding to faces $F$ of $P$ which are not facets and/or additional columns corresponding to points $v$ of $P$ that are not vertices. This fact is used in the proof of our lower bounds on extension complexity, starting with Theorem~\ref{thm:LB_CUT}.

This \emph{factorization theorem} led Yannakakis to explore connections between EFs and communication complexity. Let $S$ denote the slack matrix of the polytope~$P$. He proved that: (i) every deterministic communication protocol of complexity $k$ computing $S$ gives rise to an EF of $P$ of size at most $2^k$; (ii) the nondeterministic communication complexity of the \emph{support matrix} of $S$ (i.e., the binary matrix that has 0-entries exactly where $S$ is~0) yields a lower bound on (the base-2 logarithm\footnote{All logarithms in this paper are in base $2$.} of) the extension complexity of $P$, or more generally, the nondeterministic communication complexity of the support matrix of every nonnegative matrix~$M$ yields a lower bound on (the base-2 logarithm of) the nonnegative rank of~$M$.%
\footnote{The classical nondeterministic communication complexity of a binary communication matrix is defined as $\lceil \log B \rceil$, where $B$ is the minimum number of monochromatic 1-rectangles that cover the matrix, see~\cite{KushilevitzNisan97}. This last quantity is also known as the \emph{rectangle covering bound}. It is easy to see that the rectangle covering bound of the support matrix of any matrix $M$ lower bounds the nonnegative rank of $M$ (see Theorem \ref{thm:nnegrkvsndetcc} below).}

\paragraph{Tighter Communication Complexity Connection}

Faenza et al. \cite{FaenzaFioriniGrappeTiwary11} proved that the \mbox{base-$2$} logarithm of the nonnegative rank of a matrix equals, up to a small additive constant, the minimum complexity of a randomized communication protocol \emph{with nonnegative outputs} that computes the matrix \emph{in expectation}. In particular, every EF of size $r$ can be regarded as such a protocol of complexity $\log r + O(1)$ bits that computes a slack matrix in expectation. 

\paragraph{The Clique vs.\ Stable Set Problem}

When $P$ is the stable set polytope $\STAB(G)$ of a graph $G$ (see Section~\ref{sec:stable_set_polytopes}), the slack matrix of $P$ contains an interesting row-induced 0/1-submatrix that is the communication matrix of the \emph{clique vs.\ stable set problem} (also known as the \emph{clique vs.\ independent set problem}): its rows correspond to the cliques and its columns to the stable sets (or independent sets) of graph $G$. The entry for a clique $K$ and stable set $S$ equals $1 - |K \cap S|$. Yannakakis \cite{Yannakakis91} gave an $O(\log^2 n)$ deterministic protocol for the clique vs.\ stable set problem, where $n$ denotes the number of vertices of $G$. This gives a $2^{O(\log^2 n)} = n^{O(\log n)}$ size EF for $\STAB(G)$ whenever the whole slack matrix is 0/1, that is, whenever $G$ is a perfect graph. 

A notoriously hard open question is to determine the communication complexity (in the deterministic or nondeterministic sense) of the clique vs.\ stable set problem. (For recent results that explain why this question is hard, see \cite{KushilevitzWeinreb2009,KushilevitzWeinreb2009b}.) The best lower bound to this day is due to \cite{HuangSudakov2012}: they obtained a $\frac{6}{5} \log n - O(1)$ lower bound. Furthermore, they state a graph-theoretical conjecture that, if true, would imply a $\Omega(\log^2 n)$ lower bound, and hence settle the communication complexity of the clique vs.\ stable set problem. Moreover it would give a worst-case $n^{\Omega(\log n)}$ lower bound on the extension complexity of stable set polytopes. However, a solution to the Huang-Sudakov conjecture seems far away.

\paragraph{Factorization Theorem for General Cones} Gouveia et al.\cite{GouveiaParriloThomas2011} generalized Yannakakis's factorization theorem to other convex cones. There, the question is to know which polytopes $P \subseteq \RR^d$ can be described via a \emph{conic extended formulation}
\begin{equation} \label{eq:conic-EF}
E x + F y = g,\ y \in C
\end{equation}
for some given closed, convex cone $C \subseteq \RR^k$. Cone $C$ is said to be nice if $C^* + F^\perp$ is closed for all faces $F$ of $C$, where $C^*$ is the \emph{dual cone} of $C$. It is known that the nonnegative orthants and the PSD cones are nice. Gouveia et al. \cite{GouveiaParriloThomas2011} prove that, in case $C$ is nice and $P$ has dimension at least~$1$, such a conic EF exists if and only if the slack matrix $S$ of $P$ admits a factorization $S = TU$ where (the transpose of) each row of $T$ is in $C^*$ and each column of $U$ is in $C$. This implies the following factorization theorem for \emph{semidefinite EFs}: the \emph{semidefinite extension complexity} of every polytope $P$ equals the \emph{PSD rank} of its slack matrix $S$ (see Theorem~\ref{thm:factorToExt}).

\subsection{Our Contribution} 

Our contribution in this paper is two-fold.

\begin{itemize}
\item 
First, inspired by earlier work~\cite{Wolf03}, we define a $2^n \times 2^n$ matrix $M=M(n)$ and show that the nonnegative rank of $M$ is $2^{\Omega(n)}$ because the nondeterministic communication complexity of its support matrix is $\Omega(n)$. The latter was proved by de Wolf \cite{Wolf03} using the well-known disjointness lower bound of Razborov \cite{Razborov92}. We use the matrix $M$ to prove a $2^{\Omega(n)}$ lower bound on the extension complexity of the cut polytope $\CUT(n)$ (Section~\ref{sec:cut_polytopes}). That is, we prove that \emph{every} EF of the cut polytope has an exponential number of inequalities. Via reductions, we infer from this: (i) an infinite family of graphs $G$ such that the extension complexity of the corresponding stable set polytope $\STAB(G)$ is $2^{\Omega(\sqrt{n})}$, where $n$ denotes the number of vertices of $G$ (Section~\ref{sec:stable_set_polytopes}); (ii) that the extension complexity of the TSP polytope $\TSP(n)$ is $2^{\Omega(\sqrt{n})}$ (Section~\ref{sec:TSP_polytopes}). 

In addition to simultaneously settling the above-mentioned open problems of Yannakakis \cite{Yannakakis91} and Rothvo\ss \cite{Rothvoss11}, our results provide a lower bound on the extension complexity of stable set polytopes that goes much beyond what is implied by the Huang-Sudakov conjecture (thanks to the fact that we consider a different part of the slack matrix). Although our lower bounds are strong, unconditional and apply to explicit polytopes that are well-known in combinatorial optimization, they have very accessible proofs. 

\item 
Second, we generalize the tight connection between linear\footnote{In this paragraph, and later in Section \ref{sec:qcp}, an EF (in the sense of the previous section) is called a \emph{linear} EF. The use of adjectives such as ``linear'', ``semidefinite'' or ``conic'' will help distinguishing the different types of EFs.} EFs and classical communication complexity found by Faenza et al. \cite{FaenzaFioriniGrappeTiwary11} to a tight connection between semidefinite EFs and quantum communication complexity.\footnote{After a first version of this paper appeared, Jain et al. \cite[Theorem~2]{JSWZ2012} have used this notion of PSD rank to characterize the number of qubits of communication between Alice and Bob needed to generate a shared probability distribution.} We show that any \emph{rank-$r$ PSD factorization} of a (nonnegative) matrix $M$ gives rise to a one-way quantum protocol computing $M$ in expectation that uses \(\log r + O(1)\) qubits and, conversely, that any one-way quantum protocol computing $M$ in 
expectation that uses \(q\) qubits results in a PSD factorization of $M$ of rank \(2^q\). Via the semidefinite factorization theorem, this yields a characterization of the semidefinite extension complexity of a polytope in terms of the minimum complexity of (one-way) quantum protocols that compute the corresponding slack matrix in expectation.\smallskip

Then, we give a complexity~$\log r + O(1)$ quantum protocol for computing a nonnegative matrix $M$ in expectation, whenever there exists a rank-$r$ matrix $N$ such that $M$ is the entry-wise square of $N$. This implies in particular that every $d$-dimensional polytope with 0/1 slacks has a semidefinite EF of size $O(d)$.\smallskip

Finally, we obtain an exponential separation between classical and quantum protocols that compute our specific matrix $M = M(n)$ 
in expectation. On the one hand, our quantum protocol gives a rank-$O(n)$ PSD factorization of $M$. On the other hand, the nonnegative rank of $M$ is $2^{\Omega(n)}$ because the nondeterministic communication complexity of the support matrix of $M$ is $\Omega(n)$.  Thus we also obtain an exponential separation between PSD rank and nonnegative rank.

\end{itemize}

We would like to point out that the lower bounds on the extension complexity of polytopes established in Section~\ref{sec:Strong_LBs} were obtained by first finding an efficient PSD factorization or, equivalently, an efficient one-way quantum communication protocol for the matrix $M = M(n)$. In this sense our classical lower bounds stem from quantum considerations somewhat similar in style to \cite{KerenidisWolf03,Aaronson06,AharonovRegev04}.  See \cite{DruckerWolf11} for a survey of this line of work.

We would also like to point out that the fact that a matrix $M$ with a rank-$r$ entrywise square-root has a PSD-rank at most $r + O(1)$, which follows from Theorem~\ref{th:qupperlowrank}, was also obtained by Gouveia, Parrilo and Thomas~\cite{GouveiaParriloThomas2011}, independently (since their results were not publicly available at the time we performed our research) and in a different context. Also, after a preprint of our paper had appeared, we learned that Klauck et al. {lee:communication} had independently found a matrix (similar but not quite the same as ours) with an exponential separation between PSD rank and nonnegative rank.

\subsection{Other Related and Subsequent Work}

Yannakakis's paper has deeply influenced the TCS community. In addition to the works cited above, it has inspired a whole series of papers on the quality of restricted \emph{approximate} EFs, such as those defined by the Sherali-Adams hierarchies and Lov\'asz-Schrijver closures starting with~\cite{AroraBollobasLovasz02} (\cite{ABLT2006} for the journal version), see, e.g., \cite{BGHMT2006,STT2007,FernandezdelaVegaMathieu2007,CMM2009,GMT2009,GMPT2010,BenabbasMagen2010}. 

After the conference version of our paper appeared, there has been a lot of follow-up work, including on approximations.  Braun et al. \cite{bfps2012} developed a general framework for studying the power of approximate EFs, independent of specific hierarchies. In particular, via lower bounds on the extension complexity of approximations of the cut polytope, they showed that linear programs for approximating Max-Clique to within a factor $n^{1/2-\epsilon}$ need size at least $2^{\Omega(n^{\epsilon})}$. Similarly, they show the existence of a spectrahedron of small size that cannot be approximated by any LP with a polynomial number of inequalities within a factor of $n^{1/2-\epsilon}$. 
Braverman and Moitra \cite{BM13} used methods from information complexity to show the same size lower bound even for approximation factor $n^{1-\epsilon}$; Braun and Pokutta \cite{BP2013} subsequently simplified and generalized their result and Braun et al. \cite{BJLP2013} show that the amortized log nonnegative rank is characterized by information. Such inapproximability results should be contrasted with H{\aa}stad's famous result~\cite{hastad:inapproxclique} that it is hard to approximate Max-Clique to within a factor $n^{1-\epsilon}$: H{\aa}stad's result gives is a lower bound for \emph{all} algorithms approximating Max-Clique and is conditional on the unproven assumption that RP $\neq$ NP, while the results of \cite{bfps2012} and \cite{BM13} are geometric statements about the nonexistence of polynomial-size extended formulations.

Braun, Fiorini and Pokutta \cite{BFP2013} analyze the average-case polyhedral complexity of the maximum stable set problem showing that the extension complexity of the stable set polytope is high for almost all graphs. Pokutta and Van Vyve \cite{PV13} proved lower bounds on extension complexity for the knapsack problem, and Avis and Tiwary \cite{AT2013} proved lower bounds for the subset-sum and three-dimensional matching problems, as well as others. Kaibel and Weltge \cite{KW13} gave a more direct proof of the lower bound for the cut polytope, via bounding the size of the largest rectangle in the slack matrix, however, they still use the same set of $2^n$ valid constraints that we use here (Lemma~\ref{lem:slack}).

Chan et al. \cite{CLRS13} prove super-polynomial lower bounds on approximate EFs for MAX CSPs (constraint satisfaction problems). In particular, they prove that every $(2-\varepsilon)$-approximate (linear) EF for Max-Cut has $n^{\Omega\left(\frac{\log n}{\log \log n}\right)}$ size. This is striking because the celebrated approximation algorithm of Goemans and Williamson \cite{GoemansWilliamson95} is based on a $\Theta(n)$-size semidefinite EF with an approximation factor of at most $1.14$. Again, the result of Chan et al. \cite{CLRS13} on Max-Cut matches the algorithmic hardness of the problem Khot et al. \cite{KhotKMO07}, which assumes the Unique Games Conjecture.

Rhothvo\ss \cite{R13} proves that the matching polytope has extension complexity $2^{\Omega(n)}$, solving the second part of the main open problem in \cite{Yannakakis91}. This is the first time such a strong bound is obtained for a polytope over which one can optimize in polynomial time. Rothvo\ss's groundbreaking result implies in particular that the extension complexity of the TSP polytope is $2^{\Omega(n)}$, thus going beyond our $2^{\Omega(\sqrt{n})}$ lower bound. 

Not much is known yet about lower bounds on \emph{semidefinite} EFs. Extending the work of Rothvo\ss \cite{Rothvoss11}, Bri\"et, Dadush and Pokutta \cite{BDP13} show that most 0/1 polytopes (i.e., polytopes that are the convex hull of a random subset of $\01^d$) need exponentially large semidefinite EFs. Fawzi and Parrilo \cite{FP13} give exponential lower bounds on the size of semidefinite EFs of explicit polytopes in a restricted setting, where the underlying cone is not the full PSD cone but rather a product of fixed-dimensional PSD cones. Lee and Theis \cite{LT12} obtain polynomial lower bounds based on the support pattern of slack matrices.

Finally, Fiorini et al. \cite{FMPT13} use the notion of conic extensions and its relation to communication complexity to study generalized probabilistic theories, which are different from the usual classical or quantum-mechanical theories, and show that all polynomially-definable 0/1-polytopes have small extension complexity with respect to the completely positive cone.

\subsection{Organization}

The discovery of our lower bounds on extension complexity crucially relied on finding the right matrix $M$ and the right polytope whose slack matrix contains~$M$. In our case, we found these through a connection with quantum communication.  However, these quantum aspects are not strictly necessary for the resulting lower bound proof itself. Hence, in order to make the main results more accessible to those without background or interest in quantum computing, we start by giving a purely classical presentation of those lower bounds. 

In Section~\ref{sec:ourmatrix} we define our matrix $M$ and lower bound the nondeterministic communication complexity of its support matrix.  In Section~\ref{sec:Strong_LBs} we embed $M$ in the slack matrix of the cut polytope in order to lower bound its extension complexity; further reductions then give lower bounds on the extension complexities of the stable~set, and TSP polytopes. In Section~\ref{sec:qcp} we establish the equivalence of PSD factorizations of a (nonnegative) matrix $M$ and one-way quantum protocols that compute $M$ in expectation, and give an efficient quantum protocol in the case where some entry-wise square root of $M$ has small rank. This is then used to provide an exponential separation between quantum and classical protocols for computing a matrix in expectation (equivalently, an exponential separation between nonnegative rank and PSD rank). Concluding remarks are given in Section~\ref{sec:concluding-remarks}.  

\section{A Simple Matrix with Large Rectangle Covering Bound} \label{sec:ourmatrix}

In this section we consider the following $2^n \times 2^n$ matrix $M = M(n)$ with rows and columns indexed by $n$-bit strings $a$ and $b$, and real nonnegative entries:
$$
M_{ab} := (1 - a^{\intercal} b)^2.
$$
Note for later reference that $M_{ab}$ can also be written as 
\begin{equation}\label{eq:Sab}
M_{ab}=1-\langle 2 \diag(a) - aa^{\intercal}, bb^{\intercal}\rangle,
\end{equation}
where $\langle \cdot,\cdot\rangle$ denotes the Frobenius inner product\footnote{The \emph{Frobenius inner product} is the component-wise inner product of two matrices. For matrices $X$ and $Y$ of the same dimensions, this equals $\tr{X^{\intercal}Y}$. When $X$ is symmetric this can also be written $\tr{XY}$.}
and $\diag(a)$ is the $n\times n$ diagonal matrix with the entries of $a$ on its diagonal. Let us verify this identity, using $a,b\in\01^n$:
\begin{align*}
1-&\langle 2 \diag(a) - aa^{\intercal}, bb^{\intercal}\rangle\\ 
&= 1 - 2\langle \diag(a),bb^{\intercal}\rangle  + \langle aa^{\intercal},bb^{\intercal}\rangle\\ 
&= 1- 2 a^{\intercal} b+(a^{\intercal} b)^2=(1 - a^{\intercal} b)^2.
\end{align*}
Let $\suppmat(M)$ be the binary support matrix of $M$, so
$$
\suppmat(M)_{ab} = \left\{
\begin{array}{ll}
1 &\text{if } M_{ab}\neq 0,\\ 
0 &\text{otherwise}.
\end{array}
\right.
$$  
For a given matrix, a {\DEF rectangle} is the Cartesian product of a set of row indices and a set of column indices. In \cite{Wolf03} it was shown that an exponential number of (monochromatic) rectangles are needed to cover all the 1-entries of the support matrix of $M$. Equivalently, the corresponding function $f:\01^n\times\01^n\rightarrow\01$ has  
nondeterministic communication complexity of $\Omega(n)$ bits. For the sake of completeness we repeat the proof here:

\begin{thm}[\cite{Wolf03}]\label{thm:coverlowerboundforM}
Every 1-monochromatic rectangle cover of $\suppmat(M)$ has size $2^{\Omega(n)}$.
\end{thm}

\begin{proof}
Let $R_1,\ldots,R_k$ be a 1-cover for~$f$, i.e., a set of (possibly overlapping) 1-monochromatic rectangles in the matrix $\suppmat(M)$ that together cover all 1-entries in $\suppmat(M)$.

We use the following result
from \cite[Example~3.22 and Section~4.6]{KushilevitzNisan97},
which is essentially due to Razborov \cite{Razborov92}:
\begin{quote}
There exist sets $A,B\subseteq\01^n\times\01^n$ and probability
distribution $\mu$ on $\01^n\times\01^n$ such that
all $(a,b)\in A$ have $a^{\intercal} b=0$,
all $(a,b)\in B$ have $a^{\intercal} b=1$,
$\mu(A)=3/4$,
and there are constants $\alpha,\delta>0$ (independent of~$n$)
such that for all rectangles~$R$,
$$
\mu(R\cap B) \geqslant \alpha\cdot \mu(R\cap A)-2^{-\delta n}.
$$
(For sufficiently large $n$, $\alpha=1/135$ and $\delta=0.017$ will do.)
\end{quote}
Since the $R_i$ are 1-rectangles, they cannot contain elements from~$B$.
Hence $\mu(R_i\cap B)= 0$ and $\mu(R_i\cap A)\leqslant2^{-\delta n}/\alpha$.
However, since all elements of~$A$ are covered by the $R_i$, we have
\[
\frac{3}{4}=\mu(A)=\mu\left(\bigcup_{i=1}^k(R_i\cap A)\right) \leqslant \sum_{i=1}^k\mu(R_i\cap A)\leqslant k\cdot \frac{2^{-\delta n}}{\alpha}.
\]
Hence $k \geqslant 2^{\Omega(n)}$.
\end{proof}

\section{Strong Lower Bounds on Extension Complexity}

\label{sec:Strong_LBs}

Here we use the matrix $M = M(n)$ defined in the previous section to prove that the (linear) extension complexity of the cut polytope of the $n$-vertex complete graph is $2^{\Omega(n)}$, i.e., every (linear) EF of this polytope has an exponential number of inequalities. Then, via reductions, we prove super-polynomial lower bounds for the stable set polytopes and the TSP polytopes. To start, let us define more precisely the slack matrix of a polytope. For a matrix $A$, let $A_i$ denote the $i$th row of $A$ and let $A^j$ denote the $j$th column of $A$.

Let $P = \{x \in \RR^d \mid Ax \leqslant b\} = \conv(V)$ be a polytope, with $A \in \RR^{m \times d}$, $b \in \RR^m$ and $V = \{v_1,\ldots,v_n\} \subseteq \RR^d$. Then \(S \in \RR_+^{m \times n}\) defined as \(S_{ij} := b_i - A_i v_j\) with \(i \in [m] := \{1,\ldots,m\}\) and \(j \in [n] := \{1,\ldots,n\}\) is the \emph{slack matrix} of \(P\) w.r.t.\ $Ax \leqslant b$ and $V$. We sometimes refer to the submatrix of the slack matrix induced by rows corresponding to facets and columns corresponding to vertices simply as \emph{the} slack matrix of $P$, denoted by~$S(P)$.

Recall that:
\begin{enumerate}
\item an \emph{extended formulation} (EF) of $P$ is a linear system in variables $(x,y)$ such that $x \in P$ if and only if there exists $y$ satisfying the system;
\item an \emph{extension} of $P$ is a polytope $Q \subseteq \RR^e$ such that there is  a linear map $\pi : \RR^e \to \RR^d$ with $\pi(Q) = P$;
\item the \emph{extension complexity} of $P$ is the minimum size (i.e., number of inequalities) of an EF of $P$.
\end{enumerate}
We denote the extension complexity of $P$ by $\xc(P)$. 

\subsection{The Factorization Theorem} \label{sec:factorization_thm}

A \emph{rank-$r$ nonnegative factorization} of a (nonnegative) matrix $M$ is a factorization $M = TU$ where $T$ and $U$ are nonnegative matrices with $r$ columns (in case of \(T\)) and $r$ rows (in case of \(U\)), respectively. The nonnegative rank of $M$, denoted by $\nnegrk(M)$, is thus simply the minimum rank among all nonnegative factorizations of $M$. Note that $\nnegrk(M)$ is also the minimum $r$ such that $M$ is the sum of $r$ nonnegative rank-$1$ matrices. In particular, the nonnegative rank of a matrix $M$ is at least the nonnegative rank of any submatrix of $M$.

The following factorization theorem was proved by Yannakakis (see also \cite{FioriniKaibelPashkovichTheis11}). It can be stated succinctly as: $\xc(P) = \nnegrk(S)$ whenever $P$ is a polytope and $S$ a slack matrix of $P$. We include a sketch of the proof for completeness and we will use the following lemma which follows easily from Farkas's Lemma \cite{SchrijverBook,Ziegler} by first showing that $\mathbf{0}^\intercal x \leqslant 1$ can be derived from the system. 

\begin{lem}
\label{lem:nonnegFarkasPolytope}
Let \(P = \{x \in \RR^d \mid Ax \leqslant b\}\) be a (possibly unbounded) polyhedron that admits a direction \(u \in \RR^d\) with \(-\infty < \min \{u^\intercal x \mid x \in P\} < \max \{u^\intercal x \mid x \in P\} < +\infty\), and  \(c^\intercal x \leqslant \delta\) a valid inequality for \(P\). Then there exist nonnegative multipliers \(\lambda \in \RR^d\) such that \(\lambda^\intercal A = c^\intercal\) and \(\lambda^\intercal b = \delta\), i.e., \(c^\intercal x \leqslant \delta\) can be derived as a nonnegative combination from \(Ax \leqslant b\). In particular, this holds whenever $P$ is a polytope of dimension at least $1$ or whenever $P$ is an unbounded polyhedron that linearly projects to a polytope of dimension at least $1$.
\end{lem}

We are ready to state Yannakakis's factorization theorem.

\begin{thm}[\cite{Yannakakis91}] \label{thm:factorization}
Let $P=\{x\in\RR^d\mid Ax\leqslant b\} = \conv (V)$ be a polytope with $\dim(P) \geqslant 1$, and let $S$ denote the slack matrix of $P$ w.r.t.\ $Ax \leqslant b$ and $V$. Then the following are equivalent for all positive integers $r$:

\begin{enumerate}[(i)]
\item $S$ has nonnegative rank at most $r$;
\item $P$ has an extension of size at most $r$ (that is, with at most $r$ facets);
\item $P$ has an EF of size at most $r$ (that is, with at most $r$ inequalities).
\end{enumerate}
\end{thm}

\begin{proof}
It should be clear that (ii) implies (iii). We prove that (i) implies (ii), and then that (iii) implies (i). 

First, consider a rank-$r^*$ nonnegative factorization $S = TU$ of the slack matrix of $P$, where $r^* \leqslant r$. Notice that we may assume that no column of $T$ is zero, because otherwise $r^*$ can be decreased. We claim that $P$ is the image of
$$
Q := \{(x,y)\in\RR^{d+r^*} \mid Ax+Ty=b,\ y\geqslant \mathbf{0}\}
$$ 
under the projection $\pi_x : (x,y) \mapsto x$ onto the $x$-space. We see immediately that $\pi_x(Q) \subseteq P$ since $Ty \geqslant \mathbf{0}$. To prove the inclusion $P \subseteq \pi_x(Q)$, it suffices to remark that for each point $v_j \in V$ the point $(v_j,U^j)$ is in $Q$ since
$$
Av_j + TU^j = Av_j + b - Av_j = b \text{ and } U^j \geqslant \mathbf{0}.
$$ 
Since no column of $T$ is zero, $Q$ is a polytope. Moreover, $Q$ has at most $r^* \leqslant r$ facets, and is thus an extension of $P$ of size at most $r$. This proves that (i) implies (ii).

Second, suppose that the system
$$
E^{=}x + F^{=}y = g^{=},\ E^{\leqslant}x + F^{\leqslant}y \leqslant g^{\leqslant}
$$
with $(x,y) \in \RR^{d+k}$ defines an EF of $P$ with at most $r$ inequalities. Let $Q \subseteq \RR^{d+k}$ denote the set of solutions to this system. Thus $Q$ is a (not necessarily bounded) polyhedron. For each point $v_j \in V$, pick $w_j \in \RR^k$ such that $(v_j,w_j) \in Q$. Because
$$
Ax \leqslant b \iff \exists y : E^{=}x + F^{=}y = g^{=},\ E^{\leqslant}x + F^{\leqslant}y \leqslant g^{\leqslant},
$$
each inequality in $Ax \leqslant b$ is valid for all points of $Q$. Let $S_Q$ be the nonnegative matrix that records the slacks of the points $(v_j,w_j)$ with respect to the inequalities of $E^{\leqslant}x + F^{\leqslant}y \leqslant g^{\leqslant}$, and then of $A x \leqslant b$. By construction, the submatrix obtained from $S_Q$ by deleting the rows corresponding to the inequalities of $E^{\leqslant}x + F^{\leqslant}y \leqslant g^{\leqslant}$ and leaving only those corresponding to the inequalities of $Ax \leqslant b$ is exactly $S$, thus $\nnegrk(S) \leqslant \nnegrk(S_Q)$. Furthermore, by Lemma~\ref{lem:nonnegFarkasPolytope} any valid inequality \(c^\intercal x \leqslant \delta\) is a nonnegative combination of inequalities of the system \(Ax \leqslant b\) and thus every row of $S_Q$ is a nonnegative combination of the first $r$ rows of $S_Q$. Thus, $\nnegrk(S_Q) \leqslant r$. Therefore, $\nnegrk(S) \leqslant r$. Hence (iii) implies~(i).
\end{proof}

\begin{remark} \label{rem:factorization}
By the factorization theorem, if polytope $P \subseteq \RR^d$ has an EF of size $r$, then its slack matrix $S$ has a nonnegative factorization $S = TU$ of rank $r$. But then $Ax+Ty=b,\ y\geqslant \mathbf{0}$ is an EF of $P$ in slack form with $d+r$ variables, $r$ inequalities and $m$ equalities, where $m$ is the number of rows in the linear description $Ax \leqslant b$ of $P$. Notice that if $m > d+r$ some of these equalities will be redundant, and that there always exists a subset of at most $d+r$ equalities defining the same subspace. By removing redundant equalities from the EF, we can assume that there are at most $d+r$ equalities in the EF.
\end{remark}


We would like to emphasize that we will not restrict the slack matrix to have rows corresponding only to the facet-defining inequalities. This is not an issue since appending rows corresponding to redundant\footnote{An inequality of a linear system is called \emph{redundant} if removing the inequality from the system does not change the set of solutions.} inequalities does not change the nonnegative rank of the slack matrix. This fact was already used in the second part of the previous proof.

Theorem~\ref{thm:factorization} shows in particular that we can lower bound the extension complexity of~$P$ by lower bounding the nonnegative rank of its slack matrix~$S$; in fact it suffices to lower bound the nonnegative rank of any submatrix of the slack matrix \(S\) corresponding to an implied system of inequalities. To that end, Yannakakis made the following connection with nondeterministic communication complexity. Again, we include the (easy) proof for completeness.

\begin{thm}[\cite{Yannakakis91}] \label{thm:nnegrkvsndetcc}
Let $M$ be any matrix with nonnegative real entries and $\suppmat(M)$ its support matrix. Then $\nnegrk(M)$ is lower bounded by the rectangle covering bound for $\suppmat(M)$.
\end{thm}
\begin{proof}
If $M = TU$ is a rank-$r$ nonnegative factorization of $M$, then $S$ can be written as the sum of $r$ nonnegative rank-$1$ matrices:
$$
S = \sum_{k=1}^r T^k U_k.
$$
Taking the support on each side, we find
\begin{eqnarray*}
\supp(S) 
&=& \displaystyle \bigcup_{k=1}^r \supp(T^kU_k)\\
&=& \displaystyle \bigcup_{k=1}^r \supp(T^k) \times \supp(U_k).
\end{eqnarray*}
Therefore, $\suppmat(M)$ has a $1$-monochromatic rectangle cover with $r$ rectangles.
\end{proof}

\subsection{Cut and Correlation Polytopes} \label{sec:cut_polytopes}

Let $K_n = (V_n,E_n)$ denote the $n$-vertex complete graph. For a set $X$ of vertices of $K_n$, we let $\delta(X)$ denote the set of edges of $K_n$ with one endpoint in $X$ and the other in its complement $\bar{X}$. This set $\delta(X)$ is known as the \emph{cut} defined by $X$. For a subset $F$ of edges of $K_n$, we let $\chi^{F} \in \mathbb{R}^{E_n}$ denote the \emph{characteristic vector} of $F$, with $\chi^F_e = 1$ if $e \in F$ and $\chi^F_e = 0$ otherwise. The \emph{cut polytope} $\CUT(n)$ is defined as the convex hull of the characteristic vectors of all cuts in the complete graph $K_n = (V_n,E_n)$. That is,
\[
\CUT(n) := \conv \{\chi^{\delta(X)} \in \mathbb{R}^{E_n} \mid X \subseteq V_n\}.
\]

We will not deal with the cut polytopes directly, but rather with 0/1-polytopes that are linearly isomorphic to them. Two polytopes are called \emph{linearly isomorphic} if one can be obtained from
the other by applying an invertible linear map. It is easy to check that if
$P_1$ and $P_2$ are linearly isomorphic then they have same number of vertices
and facets. Furthermore, any extended formulation for one can be converted to an
extended formulation of the other using the same transformation. So any bound on
the extension complexity of one polytope applies to any other polytope that is
linearly isomorphic to it. The \emph{correlation polytope} (or \emph{Boolean
quadric polytope}) $\COR(n)$ is defined as the convex hull of all the rank-$1$
binary symmetric matrices of size $n \times n$. In other words, 

\[
\COR(n) := \conv \{bb^{\intercal} \in \mathbb{R}^{n \times n} \mid b \in \{0,1\}^n\}.
\]
We use the following known result:

\begin{thm}[\cite{DeSimone90}]
\label{thm:DeSimone}
For all $n$, $\COR(n)$ is linearly isomorphic to $\CUT(n+1)$. 
\end{thm}

Consider the matrix $M$ defined in Section \ref{sec:ourmatrix}. Because $M$ is nonnegative, Eq.~\eqref{eq:Sab} gives us a linear inequality that is satisfied by all vertices $bb^{\intercal}$ of $\COR(n)$, and hence (by convexity) is satisfied by all points of $\COR(n)$:

\begin{lem}
\label{lem:slack}
For all $a \in \{0,1\}^n$, the inequality
\begin{equation}
\label{eq:COR}
\langle 2 \diag(a) - aa^{\intercal}, x\rangle \leqslant 1
\end{equation}
is valid for $\COR(n)$. Moreover, the slack of vertex $x = bb^{\intercal}$ with respect to \eqref{eq:COR} is precisely $M_{ab}$. 
\end{lem}

We remark that~\eqref{eq:COR} is weaker than the \emph{hypermetric inequality} \cite{DezaLaurentBook} $\langle \diag(a)-aa^{\intercal},x \rangle \leqslant 0$, in the sense that the face defined by the former is strictly contained in the face defined by the latter. Nevertheless, we persist in using~\eqref{eq:COR}. Now, we prove the main result of this section.

\begin{thm}
\label{thm:LB_CUT}
There exists some constant $C > 0$ such that, for all $n$, 
$$
\xc(\CUT(n+1)) = \xc(\COR(n)) \geqslant 2^{Cn}\ .
$$
In particular, the extension complexity of $\CUT(n)$ is $2^{\Omega(n)}$.
\end{thm}
\begin{proof}
The equality is implied by Theorem~\ref{thm:DeSimone}. Now, consider any system of linear inequalities describing $\COR(n)$ starting with the $2^n$ inequalities~\eqref{eq:COR}, and a slack matrix $S$ w.r.t.\ this system and $\{bb^{\intercal} \mid b \in \{0,1\}^n\}$. Next delete from this slack matrix all rows except the $2^n$ first rows. By Lemma \ref{lem:slack}, the resulting $2^n\times 2^n$ matrix is~$M$. Using Theorems~\ref{thm:factorization}, \ref{thm:nnegrkvsndetcc}, and~\ref{thm:coverlowerboundforM},  and the fact that the nonnegative rank of a matrix is at least the nonnegative rank of any of its submatrices, we have
\begin{eqnarray*}
\xc(\COR(n)) & =    & \nnegrk(S)\\
             & \geqslant &\nnegrk(M)\\
             & \geqslant & 2^{Cn}
\end{eqnarray*}
for some positive constant $C$. 
\end{proof}

In their follow-up work, Kaibel and Weltge \cite{KW13} proved that one can take $C = \log(3/2) \approx 0.58$.

\subsection{Stable Set Polytopes} \label{sec:stable_set_polytopes}

A {\DEF stable set} $S$ (also called an {\DEF independent set}) of a graph $G = (V,E)$ is a subset $S \subseteq  V$ of the vertices such that no two of them are adjacent. For a subset $S \subseteq V$, we let $\chi^{S} \in \mathbb{R}^{V}$ denote the \emph{characteristic vector} of $S$, with $\chi^S_v = 1$ if $v \in S$ and $\chi^S_v = 0$ otherwise. The {\DEF stable set polytope}, denoted $\STAB(G)$, is the convex hull of the characteristic vectors of all stable sets in $G$, i.e., 
\[
\STAB(G) := \conv\{\chi^S \in\mathbb{R}^{V} \mid S \text{ stable set of } G\}.
\]

Recall that a polytope $Q$ is an extension of a polytope $P$ if $P$ is the image of $Q$ under a linear projection.

\begin{lem}\label{lem:stab_embedding}
For each $n$, there exists a graph $H_n$ with $O(n^2)$ vertices such that $\STAB(H_n)$ contains a face that is an extension of $\COR(n).$
\end{lem}
\newcommand{\e}[1]{#1}
\newcommand{\ee}[1]{\overline{#1}}
\newcommand{\eee}[1]{\underline{#1}}
\newcommand{\eeee}[1]{\overline{\underline{#1}}}
\begin{proof}
Consider the complete graph $K_n$ with vertex set $V_n := [n]$. For each vertex $i$ of $K_{n}$ we create two vertices labeled $\e{ii}, \ee{ii}$ in $H_n$ and an edge between them. Let us label the edges of $K_n$ in the following way. The edge between vertices $i$ and $j$ with $i<j$ gets the label $ij$. Now, for each edge $ij$ of $K_{n},$ we add to $H_n$ four vertices labeled $\e{ij}, \ee{ij}, \eee{ij}, \eeee{ij}$ and all possible six edges between them. We further add the following eight edges to $H_n$:
\begin{align*}
\{\e{ij},\ee{ii}\}, \{\e{ij},\ee{jj}\},
\{\ee{ij},\e{ii}\}, \{\ee{ij},\ee{jj}\},\\
\{\eee{ij},\ee{ii}\}, \{\eee{ij},\e{jj}\},
\{\eeee{ij},\e{ii}\}, \{\eeee{ij},\e{jj}\}.
\end{align*}
See Fig.~\ref{fig:stab_embedding} for an illustration. The number of vertices in $H_n$ is $2n+4{n\choose 2}.$

\begin{figure}[!htb]
  \centering
  \includegraphics[width=0.35\textwidth]{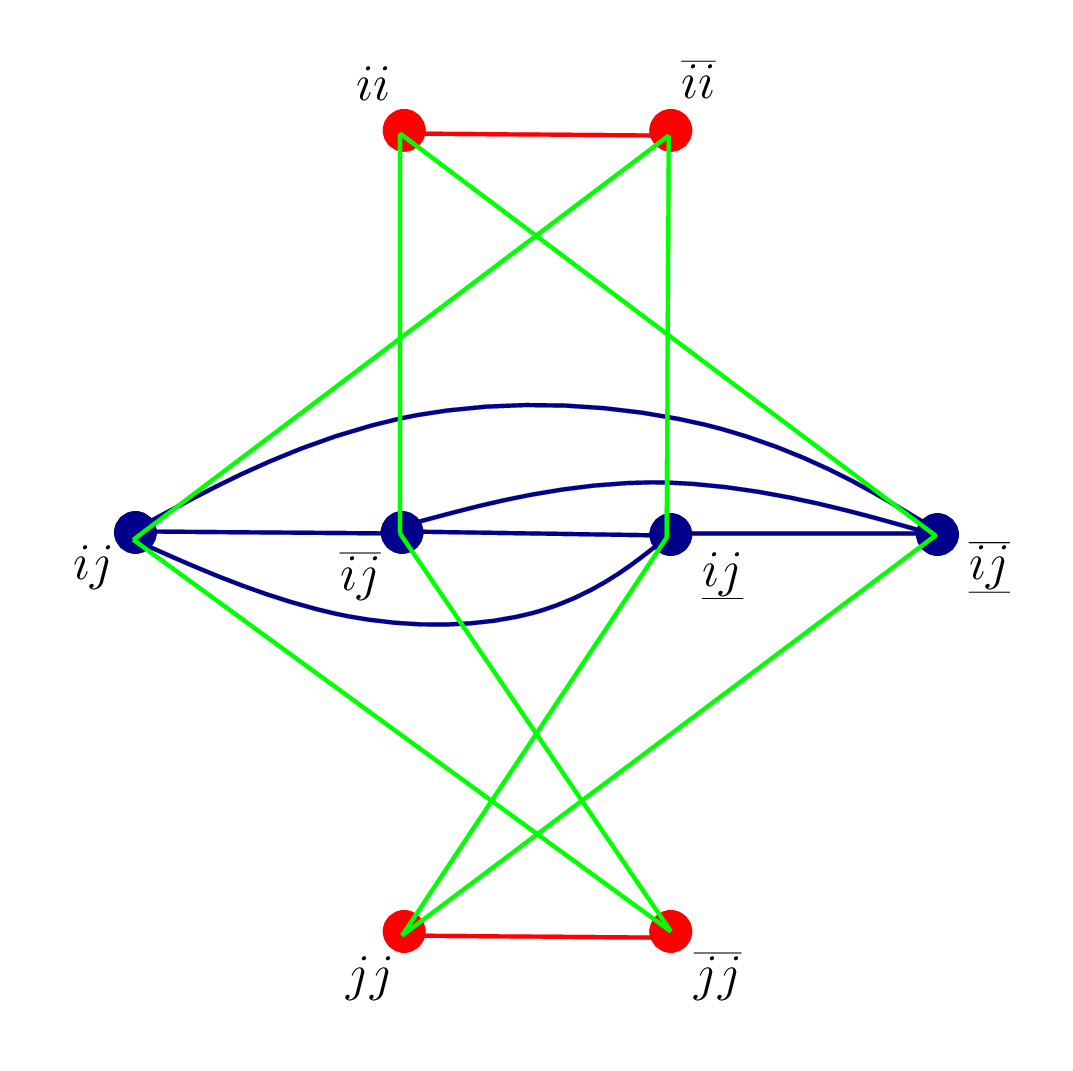}
  \caption{The edges and vertices of $H_n$ corresponding to vertices $i$, $j$ and edge $ij$ of $K_{n}.$}
  \label{fig:stab_embedding}
\end{figure}

Thus the vertices and edges of $K_{n}$ are represented by cliques of size $2$ and $4$ respectively in $H_n$. We will refer to these as \emph{vertex-cliques} and \emph{edge-cliques} respectively. Consider the face $F = F(n)$ of $\STAB(H_n)$ whose vertices correspond to the stable sets containing exactly one vertex in each vertex-clique and each edge-clique. (The vertices in this face correspond exactly to stable sets of $H_n$ with maximum cardinality.) 

Consider the linear map $\pi : \RR^{V(H_n)} \to \RR^{n \times n}$ mapping a point $x \in \RR^{V(H_n)}$ to the point $y \in \RR^{n \times n}$ such that $y_{ij} = y_{ji} = x_{ij}$ for $i \leqslant j$. In this equation, the subscripts in $y_{ij}$ and $y_{ji}$ refer to an ordered pair of elements in $[n]$, while the subscript in $x_{ij}$ refers to a vertex of $H_n$ that corresponds either to a vertex of $K_{n}$ (if $i=j$) or to an edge of $K_{n}$  (if $i\neq j$).

We claim that the image of $F$ under $\pi$ is $\COR(n)$, hence $F$ is an extension of $\COR(n)$; observe that is suffices to consider 0/1 vertices as \(F\) is a 0/1 polytope and the projection is an orthogonal projection. Indeed, pick an arbitrary stable set $S$ of $H_n$ such that $x:=\chi^S$ is on face~$F$.
Then define $b \in \{0,1\}^{n}$ by letting $b_i := 1$ if $\e{ii} \in S$ and $b_i := 0$ otherwise (i.e., $\ee{ii} \in S$). Notice that for the edge $ij$ of $K_{n}$ we have $\e{ij} \in S$ if and only if both vertices $ii$ and $jj$ belong to $S$. Hence, $\pi(x) = y = bb^{\intercal}$ is a vertex of $\COR(n)$. This proves $\pi(F) \subseteq \COR(n)$. Now pick a vertex $y := bb^{\intercal}$ of $\COR(n)$ and consider the unique maximum stable set $S$ that contains vertex $\e{ii}$ if $b_i = 1$ and vertex $\ee{ii}$ if $b_i = 0$. Then $x := \chi^S$ is a vertex of $F$ with $\pi(x) = y$. Hence, $\pi(F) \supseteq \COR(n)$. Thus $\pi(F) = \COR(n)$. This concludes the proof. 
\end{proof}

Our next lemma establishes simple monotonicity properties of the extension complexity used in our reduction.

\begin{lem} \label{lem:monotone}
Let $P$, $Q$ and $F$ be polytopes. Then the following hold:
\begin{enumerate}[(i)]
\item if $F$ is an extension of $P$, then $\xc(F) \geqslant \xc(P)$;
\item if $F$ is a face of $Q$, then $\xc(Q) \geqslant \xc(F)$.
\end{enumerate}
\end{lem}
\begin{proof}
The first part is obvious because every extension of $F$ is in particular an extension of $P$. For the second part, notice that a slack matrix of $F$ can be obtained from the (facet-vs-vertex) slack matrix of $Q$ by deleting columns corresponding to vertices not in $F$. Now apply Theorem \ref{thm:factorization}.
\end{proof}

Using previous results, we can prove the following result about the worst-case extension complexity of the stable set polytope.

\begin{thm}
\label{thm:LB_STAB}
For all $n$, one can construct a graph $G_n$ with $n$ vertices such that the extension complexity of the stable set polytope $\STAB(G_n)$ is $2^{\Omega(\sqrt{n})}$.
\end{thm}
\begin{proof}
W.l.o.g., we may assume $n \geqslant 18$. For an integer $p \geqslant 3$, let $f(p) := |V(H_{p})| = 2p + 4 {p \choose 2}$. Given $n \geqslant 18$, we define $p$ as the largest integer with  $f(p) \leqslant n$. Now let $G_n$ be obtained from $H_{p}$ by adding $n - f(p)$ isolated vertices. Then $\STAB(H_{p})$ is linearly isomorphic to a face of $\STAB(G_n)$. Using Theorem \ref{thm:LB_CUT} in combination with Lemmas \ref{lem:stab_embedding} and \ref{lem:monotone}, we find that
\begin{eqnarray*}
\xc(\STAB(G_n)) 
&\geqslant &\xc(\STAB(H_{p}))\\ 
&\geqslant &\xc(\COR(p))\\
&=& 2^{\Omega(p)}\\
&=& 2^{\Omega(\sqrt{n})}.
\end{eqnarray*}
\hfill \ 
\end{proof}

\subsection{TSP Polytopes} \label{sec:TSP_polytopes}

Recall that $\TSP(n)$, the \emph{traveling salesman polytope} or \emph{TSP polytope} of $K_n = (V_n,E_n)$, is defined as the convex hull of the characteristic vectors of all subsets $F \subseteq E_n$ that define a tour of $K_n$. That is,
\[
\TSP(n) := \conv \{\chi^F \in \mathbb{R}^{E_n} \mid F \subseteq E_n \text{ is a tour of } K_n\}.
\]

We now prove that the polytope $\COR(n)$ is the linear projection of a face of $\TSP(O(n^2))$, implying the following:

\begin{lem} \label{lem:tsp_embedding}
For each $n$, there exists a positive integer $q = O(n^2)$ such that $\TSP(q)$ contains a face that is an extension of $\COR(n)$.
\end{lem}

\begin{proof}
Recall that 
\[
\COR(n) = \conv \{bb^{\intercal} \in \mathbb{R}^{n \times n} \mid b \in \{0,1\}^n\}.
\]
To prove the lemma we start with constructing a graph $G_n$ with $q = O(n^2)$ vertices such that the tours of $G_n$ correspond to the $n \times n$ rank-$1$ binary symmetric matrices $bb^\intercal$, where $b \in \{0,1\}^n$. This is done in three steps:

\begin{enumerate}[(i)]
\item define a 3SAT formula $\phi_n$ with $n^2$ variables such that the satisfying assignments of $\phi_n$ bijectively correspond to the matrices $bb^\intercal$, where $b \in \{0,1\}^n$;
\item construct a directed graph $D_n$ with $O(n^2)$ vertices such that each directed tour of $D_n$ defines a satisfying assignment of $\phi_n$, and conversely each satisfying assignment of $\phi_n$ has at least one corresponding directed tour in $D_n$;
\item modify the directed graph $D_n$ into an undirected graph $G_n$ in such a way that the tours of $G_n$ bijectively correspond to the directed tours of $D_n$.
\end{enumerate}

\emph{Step (i)}. For defining $\phi_n$ we use Boolean variables $C_{ij} \in \{0,1\}$ for $i, j \in [n]$ and let
$$
\phi_n := \bigwedge_{i,j \in [n] \atop i \neq j} \left[
(C_{ii} \lor C_{jj} \lor \overline{C_{ij}}) \land  (C_{ii}\lor \overline{C_{jj}} \lor \overline{C_{ij}}) \land (\overline{C_{ii}} \lor C_{jj} \lor \overline{C_{ij}}) \land (\overline{C_{ii}} \lor \overline{C_{jj}} \lor C_{ij}) \right].
$$
The four clauses $(C_{ii} \lor C_{jj} \lor \overline{C_{ij}})$, $(C_{ii}\lor \overline{C_{jj}} \lor \overline{C_{ij}})$, $(\overline{C_{ii}} \lor C_{jj} \lor \overline{C_{ij}})$ and $(\overline{C_{ii}} \lor \overline{C_{jj}} \lor C_{ij})$ model the equation $C_{ij} = C_{ii} \land C_{jj}$. Hence, $C \in \{0,1\}^{n \times n}$ satisfies $\phi_n$ if and only if there exists $b \in \{0,1\}^n$ such that $C_{ij} = b_i \land b_j$ for all $i, j \in [n]$, or in matrix language, $C = bb^\intercal$.

\emph{Step (ii)}. To construct a directed graph $D_n$ whose directed tours correspond to the satisfying assignments of $\phi_n$ we use the standard reduction from 3SAT to HAMPATH \cite{SipserBook}.

We order the variables of $\phi_n$ arbitrarily and construct a gadget for each variable as follows. Suppose that the $k$th variable occurs in $p$ clauses. We create a chain of $3p+1$ nodes, labeled $v_{k,1}$,\ldots, $v_{k,3p+1}$, where each node $v_{k,\ell}$ with $\ell < 3p+1$ is connected to the next node $v_{k,\ell+1}$ with two opposite directed edges. Figure \ref{fig:variable_gadget} illustrates this. Traversing this chain from left to right is interpreted as setting the $k$th variable to false and traversing it from right to left is interpreted as setting the $k$th variable to true.  We also have two nodes $s_k, t_k$ connected to this chain with directed edges $(s_k, v_{k,1})$, $(s_k,v_{k,3p+1})$, $(v_{k,1},t_k)$ and $(v_{k,3p+1},t_k)$ creating a diamond structure.

\begin{figure}[!htb]
  \centering
  \includegraphics[width=0.45\textwidth]{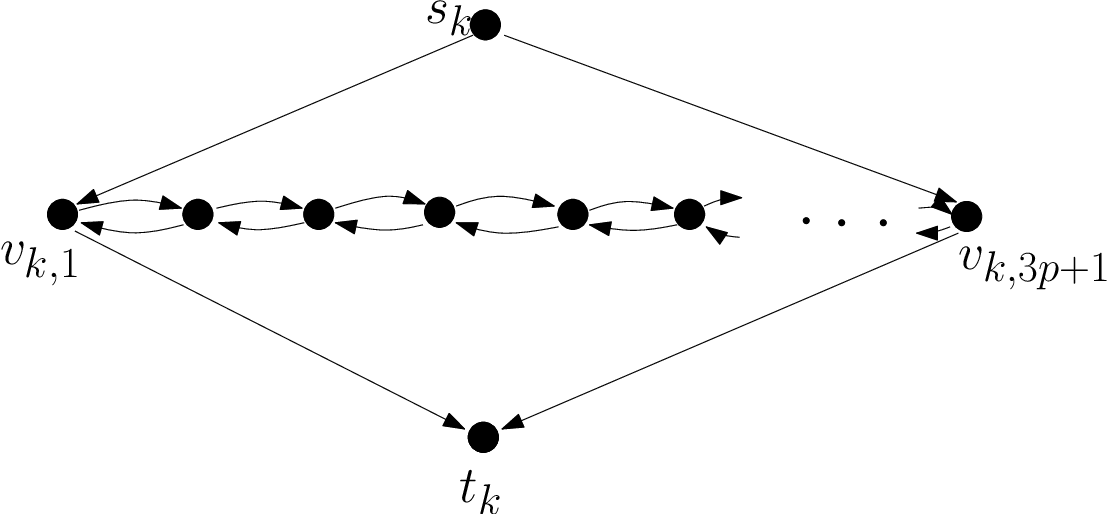}
  \caption{Gadget for the $k$th variable occurring in $p$ clauses.}
  \label{fig:variable_gadget}
\end{figure}

Next, we order the clauses of $\phi_n$ arbitrarily and create a node for each clause. The node for the $m$th clause is denoted by $w_m$. We connect these extra nodes to the gadgets for the variables as follows. Suppose, as before, that the $k$th variable appears in $p$ clauses. Consider the $\ell$th of these clauses in which the $k$th variable appears, and let $m$ be the index of that clause. If the $k$th variable appears negated in the $m$th clause then we add the path $v_{k,3\ell-1}, w_m, v_{k,3\ell}$. Otherwise the $k$th variable appears non-negated in the $m$th clause and we add the path $v_{k,3\ell}, w_m, v_{k,3\ell-1}$. Figure \ref{fig:clause_gadget} illustrates this.

\begin{figure}[!htb]
  \centering
  \includegraphics[width=0.45\textwidth]{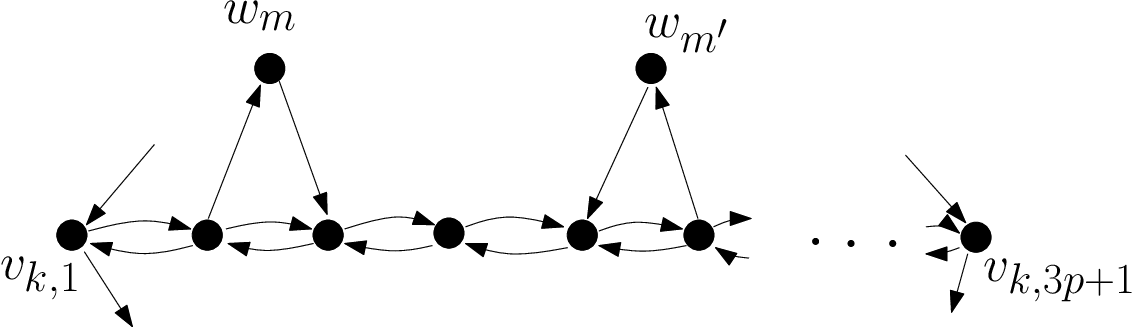}
  \caption{Gadgets for clauses in case the $k$th variable appears negated in the $m$th clause and non-negated in the $m'$th clause.}
  \label{fig:clause_gadget}
\end{figure}

Next we connect the gadgets corresponding to the variables by identifying $t_k$ with $s_{k+1}$ for $1 \leqslant k < n^2$. Finally, we add a directed edge from $t_{n^2}$ to $s_1$. Figure \ref{fig:final_gadget} illustrates the final directed graph obtained.

\begin{figure}[!htb]
  \centering
  \includegraphics[width=0.45\textwidth]{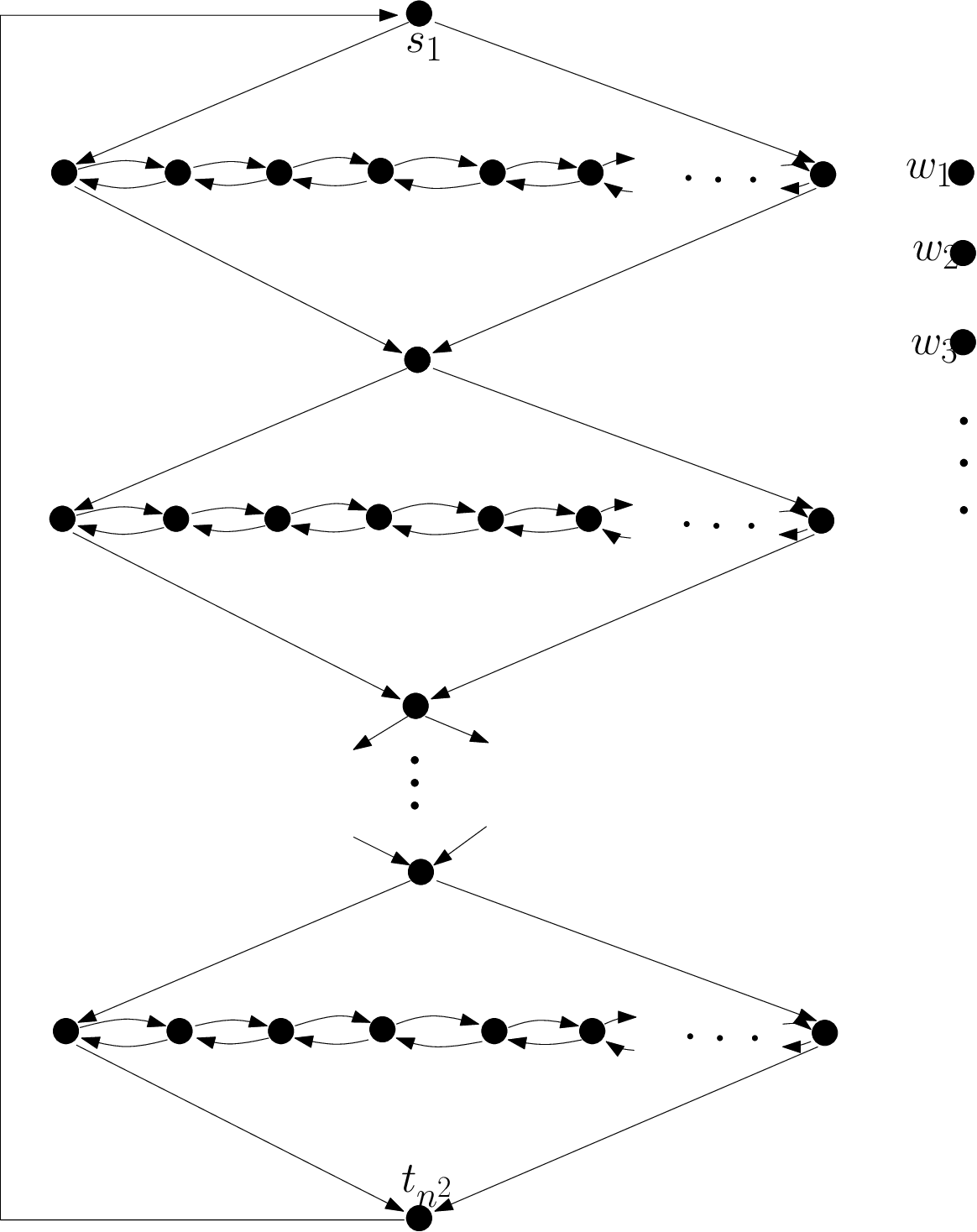}
  \caption{Final graph. Directed edges incident to nodes $w_m$ depend on the actual ordering of variables and clauses in Boolean formula $\phi_n$.}
  \label{fig:final_gadget}
\end{figure}

To see why the directed tours of the final directed graph $D_n$ define satisfying assignments of our Boolean formula $\phi_n$, observe that each directed tour of $D_n$ encodes a truth assignment to the $n^2$ variables depending on which way the corresponding chains are traversed. Because a directed tour visits every node and because the node $w_m$ corresponding to a clause can be visited only if we satisfy it, the truth assignment satisfies $\phi_n$. Conversely, every satisfying assignment of $\phi_n$ yields at least one directed tour in $D_n$. (If the $m$th clause is satisfied by the value of more than one variable, we visit $w_m$ only once, from the chain of the first variable whose value makes the clause satisfied).
  
\emph{Step (iii)}. For each node $v$ of $D_n$ we create a path $v_\mathrm{in}, v_\mathrm{mid}, v_\mathrm{out}$ in the (undirected) graph $G_n$. For each directed edge $(v,w)$ of $D_n$, we add to the graph $G_n$ an edge between $v_\mathrm{out}$ and $w_\mathrm{in}$. As is easily seen, the tours of $G_n$ bijectively correspond to the directed tours of $D_n$. Note that $G_n$ has $q := 3(n(n-1) \cdot 13 + n \cdot (3n-2) + n^2 + 4n(n-1)) = O(n^2)$ vertices.

Consider the face $F$ of $\TSP(q)$ defined by setting to $0$ all variables $x_{e}$ corresponding to non-edges of $G_n$, so that the vertices of $F$ are the characteristic vectors of the tours of $G_n$. To conclude the proof we give a linear projection $\pi : x \mapsto y := \pi(x)$ mapping $F$ to $\COR(n)$. For $x \in \RR^{E_q}$ and $i,j \in [n]$, we let $y_{ij} = x_e$, where $e$ is the edge \((v_{out,k,2}, v_{in,k,1})\) of $G_n$ corresponding to the directed edge $(v_{k,2},v_{k,1})$ and $k$ is the index of the variable $C_{ij}$ of $\phi_n$. It follows from the above discussion that $\pi$ maps the face $F$ of $\TSP(q)$ to $\COR(n)$. The lemma follows.
\end{proof}

The final theorem in this section follows from Theorem \ref{thm:LB_CUT}, Lemmas \ref{lem:monotone} and \ref{lem:tsp_embedding}, using an argument similar to that used in the proof of Theorem~\ref{thm:LB_STAB}.

\begin{thm}
\label{thm:LB_TSP}
The extension complexity of the TSP polytope $\TSP(n)$ is $2^{\Omega(n^{1/2})}$.
\end{thm}

Throughout this paper we use $M_i$ and $M^j$ to denote, respectively, the $i$-th row and $j$-th column of $M.$ For convenience we define \([n]:= \{1, \dots,n\}\) for \(n \in \mathbb N\).

\begin{eqnarray*}
T_i U^j &= &  \left(\sum_{k=1}^f \mu_{i,k} T_k\right) \left(\sum_{\ell=1}^v \lambda_{j,\ell} U^\ell\right)\\
        &= & \sum_{k=1}^f \sum_{\ell=1}^v \mu_{i,k} \lambda_{j,\ell} T_k U^\ell\\
        &= & \sum_{k=1}^f \sum_{\ell=1}^v \mu_{i,k} \lambda_{j,\ell} (b_k - A_k v_\ell)\\
        &= & \sum_{k=1}^f \mu_{i,k} b_k - 
             \left(\sum_{k=1}^f \mu_{i,k} A_k\right) \left(\sum_{\ell=1}^v \lambda_{j,\ell} v_\ell\right)\\
        &= & b_i - A_i v_j\\
        &= & S_{ij}.
\end{eqnarray*}

\section{Quantum Communication and PSD Factorizations}\label{sec:qcp}

In this section we explain the connection with quantum communication.  This yields results that are interesting in their own right, and also clarifies where the matrix $M$ of Section~\ref{sec:ourmatrix} came from.

For a general introduction to quantum computation we refer to \cite{NielsenChuangBook} and to \cite{mermin2007quantum}, and 
for quantum communication complexity we refer to \cite{Wolf02} and to \cite{BCMW10}.
For our purposes, an $r$-dimensional \emph{quantum state} $\rho$ is an $r \times r$ PSD matrix of trace~1.\footnote{For simplicity we restrict to real rather than complex entries, which does not significantly affect the results.}
A \emph{$k$-qubit state} is a state in dimension $r=2^k$.
If $\rho$ has rank~\(1\), it can be written as an outer product $\ketbra{\phi}{\phi}$ of some unit column vector $\ket{\phi}$ and its conjugate transpose $\bra{\phi}$ (which is a row vector). This $\ket{\phi}$ is sometimes called a \emph{pure state}. 
We use $\ket{i}$ to denote the pure state vector that has~\(1\) at position~$i$ and \(0\)s elsewhere.
A quantum measurement (POVM) is described by a set of PSD matrices $\{E_\theta\}_{\theta\in\Theta}$, 
each labeled by a real number $\theta$, and summing to the $r$-dimensional identity: $\sum_{\theta\in\Theta} E_\theta=I$.
When measuring state $\rho$ with this measurement, the probability of outcome $\theta$ equals $\tr{E_\theta\rho}$.
Note that if we define the PSD matrix $E:=\sum_{\theta\in\Theta} \theta E_\theta$, then the \emph{expected value}
of the measurement outcome is $\sum_{\theta\in\Theta}\theta\tr{E_\theta\rho}=\tr{E\rho}$.

\subsection{PSD Factorizations}

Analogous to nonnegative factorizations and nonnegative rank, one can define PSD factorizations and PSD rank. A \emph{rank-$r$ PSD factorization} of an $m \times n$ matrix $M$ is a collection of $r \times r$ symmetric positive semidefinite matrices $T_1, \ldots, T_m$ and $U^1, \ldots, U^n$ such that the Frobenius product $\langle T_i, U^j \rangle = \tr{(T_i)^{\intercal}U^j} = \tr{T_i U^j}$ equals $M_{ij}$ for all $i \in [m], j \in [n]$. The \emph{PSD rank} of $M$ is the minimum $r$ such that $M$ has a rank-$r$ PSD factorization. We denote this by $\psdrk(M)$.

Below, we show that $\psdrk(M)$ can be expressed in terms of the amount of communication needed by a one-way quantum communication protocol for computing $M$ in expectation (Corollary~\ref{cor:fullEquiv}). Before doing so, we state the geometric interpretation of $\psdrk(M)$ when $M$ is a slack matrix.

For a positive integer $r$, we let $\mathbb{S}^r_+$ denote the cone of $r \times r$ symmetric positive semidefinite matrices embedded in $\RR^{r(r+1)/2}$ in such a way that, for all $y, z \in \mathbb{S}^r_+$, the scalar product $z^{\intercal}y$ is the Frobenius product of the corresponding matrices. A \emph{semidefinite EF of size $r$} is a conic EF w.r.t.\ $C = \mathbb{S}^r_+$, that is, a system $Ef + Fy = g$, $y \in \mathbb{S}^r_+$ such that $P = \{x \in \mathbb{R}^d \mid \exists y : Ef + Fy = g$, $y \in \mathbb{S}^r_+\}$. We call the set $Q=\{(x,y)\in\RR^{d+r(r+1)/2}\mid Ex+Fy=g,\ y \in \mathbb{S}^r_+\}$ a \emph{semidefinite extension of $P$}. The \emph{semidefinite extension complexity} of polytope $P$, denoted by $\xcp(P)$, is the minimum $r$ such that $P$ has a semidefinite EF of size $r$. Observe that $(\mathbb{S}^r_+)^* = \mathbb{S}^r_+$. 

The following result follows from \cite{GouveiaParriloThomas2011}:

\begin{thm}
\label{thm:factorToExt}
Let $P=\{x\in\RR^d\mid Ax\leqslant b\} = \conv (V)$ be a polytope of dimension at least~$1$. Then the slack matrix $S$ of $P$ w.r.t.\ $Ax \leqslant b$ and $V$ has a factorization $S = \lfactor\rfactor$ so that $(\lfactor_i)^{\intercal}, \rfactor^j\in \mbS^r_+$ if and only if there exists a semidefinite extension $Q=\{(x,y)\in\RR^{d+r(r+1)/2}\mid Ex+Fy=g,\ y \in \mathbb{S}^r_+\}$ with $P=\pi_x(Q)$. 
\end{thm}

\subsection{Quantum Protocols}

A \emph{one-way quantum protocol with $r$-dimensional messages} can be described as follows.
On input~$i$, Alice sends Bob an $r$-dimensional state $\rho_i$.
On input~$j$, Bob measures the state he receives with a POVM $\{E^j_\theta\}$ for some \emph{nonnegative} values $\theta$,
and outputs the result. We say that such a protocol \emph{computes a matrix $M$ in expectation}, 
if the expected value of the output on respective inputs $i$ and $j$, equals the matrix entry $M_{ij}$. 
Analogous to the equivalence between classical protocols and nonnegative factorizations of~$M$ 
established by Faenza et al. \cite{FaenzaFioriniGrappeTiwary11},
such quantum protocols are essentially equivalent to PSD factorizations of~$S$:

\begin{thm}
\label{thm:qProtImpliesFact}
Let $M \in \RR_+^{m\times n}$ be a matrix. Then the following holds:
\begin{enumerate}[(i)]
\item A one-way quantum protocol with $r$-dimensional messages that computes $M$ in expectation, gives a rank-$r$ PSD factorization of $M$.
\item A rank-$r$ PSD factorization of $M$ gives a one-way quantum protocol with $(r+1)$-dimensional messages that computes $M$ in expectation.
\end{enumerate}
\end{thm}

\begin{proof}
The first part is straightforward.  Given a quantum protocol as above, define $E^j:=\sum_{\theta\in\Theta} \theta E^j_\theta$.
Clearly, on inputs $i$ and $j$ the expected value of the output is $\tr{\rho_i E^j}=M_{ij}$.

For the second part, suppose we are given a PSD factorization of a matrix $M$,
so we are given PSD matrices $T_1,\ldots,T_m$ and $U^1,\ldots,U^n$ satisfying $\tr{T_iU^j}=M_{ij}$ for all $i,j$.
In order to turn this into a quantum protocol, define $\tau=\max_i\tr{T_i}$.
Let $\rho_i$ be the $(r+1)$-dimensional quantum state obtained by adding a $(r+1)$st row and column to $T_i/\tau$,
with $1-\tr{T_i}/\tau$ as $(r+1)$st diagonal entry, and \(0\)s elsewhere.  
Note that $\rho_i$ is indeed a PSD matrix of trace~\(1\), so it is a well-defined quantum state.
For input $j$, derive Bob's $(r+1)$-dimensional POVM from the PSD matrix $U^j$ as follows.
Let $\lambda$ be the largest eigenvalue of $U^j$, and 
define $E^j_{\tau\lambda}$ to be $U^j/\lambda$, extended with a $(d+1)$st row and column of 0s. 
Let $E^j_0=I-E^j_{\tau\lambda}$. 
This is positive semidefinite because the largest eigenvalue of $E^j_{\tau\lambda}$ is~1. Hence the two operators $E^j_{\tau\lambda}$  and $E^j_0$ together form a well-defined POVM.
The expected outcome (on inputs $i,j$) of the protocol induced by the states and POVMs that we just defined,
is 
$$
\tau\lambda\tr{E^j_{\tau\lambda}\rho_i}=\tr{T_iU^j}=M_{ij},
$$ 
so the protocol indeed computes $M$ in expectation.
\end{proof}

We obtain the following corollary which summarizes the characterization of semidefinite EFs:

\begin{cor}
  \label{cor:fullEquiv}
For a polytope $P$ with slack matrix $S$, the following are equivalent:
\begin{enumerate}[(i)]
\item $P$ has a semidefinite extension $Q=\{(x,y)\in\RR^{d+r(r+1)/2}\mid Ex+Fy=g, y \in \mathbb{S}^r_+\}$;
\item the slack matrix $S$ has a rank-$r$ PSD factorization;
\item there exists a one-way quantum communication protocol with $(r+1)$-dimensional messages (i.e., using $\lceil \log (r+1) \rceil$ qubits) that computes \(S\) in expectation (for the converse we consider \(r\)-dimensional messages). 
\end{enumerate}
\end{cor}

\subsection{A General Upper Bound on Quantum Communication}

Now we provide a quantum protocol that efficiently computes a nonnegative matrix $M$ in expectation, whenever there is a low rank matrix $N$ whose entry-wise square is $M$.

\begin{thm}\label{th:qupperlowrank}
Let $M$ be a matrix with nonnegative real entries, $N$ be a rank-$r$ matrix of the same dimensions such that $M_{ij}=N^2_{ij}$. Then there exists a one-way quantum protocol using $(r+1)$-dimensional pure-state messages that computes $M$ in expectation.
\end{thm}

\begin{proof}
By Corollary~\ref{cor:fullEquiv}, it suffices to give a rank-$r$ PSD factorization of $M$.
To this end, let $t_i,u_j$ be $r$-dimensional real vectors such that $N_{ij}=t_i^\intercal u_j$; such vectors exist because~$N$ has rank~$r$.
Define $r\times r$ PSD matrices $T_i:=t_it_i^\intercal$ and $U^j:=u_ju_j^\intercal$. Then
$$
\tr{T_iU^j}=(t_i^\intercal u_j)^2=N^2_{ij}=M_{ij},
$$
hence we have a rank-$r$ PSD factorization of $M$.
%
\end{proof}

Note that if $M$ is a 0/1-matrix then we may take $N=M$, hence any low-rank 0/1-matrix can be computed in expectation by an efficient quantum protocol. 
If this~$M$ is the slack matrix for a polytope $P\subseteq\mathbb{R}^d$, then it is easy to see that its rank is at most $d+1$: the slack $M_{ij}=b_i-A_iv_j$ of a constraint $A_i x\leqslant b_i$ w.r.t.\ a point $v_j \in P$ can be written as the inner product between the two $(d+1)$-dimensional vectors $(b_i,-A_i)$ and~$(1,v_j)$. We thus obtain the following corollary (implicit in Theorem 4.2 of \cite{GouveiaParriloThomas2010}) which also implies a compact (i.e., polynomial size) semidefinite EF for the stable set polytope of perfect graphs, reproving the previously known result by Lov\'asz~\cite{Lovasz79,Lovasz03}. We point out that the result still holds when $\dim(P)+2$ is replaced by $\dim(P)+1$, see \cite{GouveiaParriloThomas2011}; this difference is due to normalization.

\begin{cor}
\label{cor:smallSDPExt}
Let \(P\) be a polytope such that \(S(P)\) is a 0/1 matrix. Then 
\(\xcp(P) \leqslant \dim(P) + 2.\) 
\end{cor}



\subsection{Quantum vs Classical Communication, and PSD vs Nonnegative Factorizations}
\label{sec:separation}

We now give an example of an exponential separation between quantum and classical communication in expectation, based on the matrix $M$ of Section~\ref{sec:ourmatrix}. This result actually preceded and inspired the results in Section~\ref{sec:Strong_LBs}.

%
%
%

\begin{thm}
  \label{thm:expSep}
For each $n$, there exists a nonnegative matrix \(M \in \RR^{2^n\times 2^n}\) that can be computed in expectation by a quantum protocol using \(\log n + O(1)\) qubits, while any classical randomized protocol needs \(\Omega(n)\) bits to compute \(M\) in expectation.
\end{thm}
%
\begin{proof}
Consider the matrix $N \in \RR^{2^n\times 2^n}$ whose rows and columns are indexed by $n$-bit strings $a$ and $b$, respectively, and whose entries are defined as $N_{ab}=1 - a^{\intercal}b$.  Define $M \in\RR_+^{2^n\times 2^n}$ by $M_{ab}=N_{ab}^2$.  
This $M$ is the matrix from Section~\ref{sec:ourmatrix}.
Note that $N$ has rank $r \leqslant n+1$ because it can be written as the sum of $n+1$ rank-\(1\) matrices. Hence Theorem~\ref{th:qupperlowrank} immediately implies a quantum protocol with $(n+2)$-dimensional messages that computes $M$ in expectation.

For the classical lower bound, note that a protocol that computes $M$ in expectation has positive probability of giving a nonzero output on input $a,b$ if and only if $M_{ab}>0$. With a message $m$ in this protocol we can associate a rectangle $R_m=A\times B$
where $A$ consists of all inputs $a$ for which Alice has positive probability of sending $m$, and $B$ consists of
all inputs $b$ for which Bob, when he receives message $m$, has positive probability of giving a nonzero output.
Together these rectangles will cover exactly the nonzero entries of $M$.
Accordingly, a $c$-bit protocol that computes $M$ in expectation 
induces a rectangle cover for the support matrix of $M$ of size $2^c$.
Theorem~\ref{thm:coverlowerboundforM} lower bounds the size of such a cover by $2^{\Omega(n)}$, hence $c=\Omega(n)$.
\end{proof}

Together with Theorem~\ref{thm:qProtImpliesFact} and the equivalence of randomized communication complexity (in expectation) and nonnegative rank established in \cite{FaenzaFioriniGrappeTiwary11}, we immediately obtain an exponential separation between nonnegative rank and PSD rank.

\begin{cor}
\label{cor:expSeparRank}
For each $n$, there exists \(M \in \RR^{2^n\times 2^n}_+\), with \(\nnegrk(M) = 2^{\Omega(n)}\) and \(\psdrk(M) = O(n)\).
\end{cor}

In fact a simple rank-$(n+1)$ PSD factorization of $M$ is the following: let $T_a:= {1 \choose -a} {1 \choose -a}^{\intercal}$ and
$U^b := {1 \choose b} {1 \choose b}^{\intercal}$, then $\tr{T_a U^b}=(1-a^{\intercal} b)^2=M_{ab}$.

\section{Concluding Remarks} \label{sec:concluding-remarks}

In addition to proving the first unconditional super-polynomial lower bounds on the size of linear EFs for the cut polytope, stable set polytope and TSP polytope, we demonstrate that the rectangle covering bound can prove strong results in the context of EFs. In particular, it can be super-polynomial in the dimension and the logarithm of the number of vertices of the polytope, settling an open problem of \cite{FioriniKaibelPashkovichTheis11}. 

The exponential separation between nonnegative rank and PSD rank that we prove here (Theorem \ref{thm:expSep}) actually implies more than a super-polynomial lower bound on the extension complexity of the cut polytope. As noted in Theorem \ref{thm:DeSimone}, the polytopes $\CUT(n)$ and $\COR(n-1)$ are affinely isomorphic. Let $Q(n)$ denote the polyhedron isomorphic (under the same affine map) to the polyhedron defined by \eqref{eq:COR} for $a \in \{0,1\}^n$. Then (i) \emph{every} polytope (or polyhedron) that contains $\CUT(n)$ and is contained in $Q(n)$ has exponential extension complexity; (ii) there exists a low complexity spectrahedron that contains $\CUT(n)$ and is contained in $Q(n)$. (A \emph{spectrahedron} is any projection of an affine slice of the positive semidefinite cone.)
This was used in \cite{bfps2012} to establish the existence of a spectrahedron that cannot be well approximated by linear programs of polynomial size. 

An important problem also left open in~\cite{Yannakakis91} is whether the perfect matching polytope has a polynomial-size linear EF. Yannakakis proved that every \emph{symmetric} EF of this polytope has exponential size, a striking result given the fact that the perfect matching problem is solvable in polynomial time. He conjectured that asymmetry also does not help in the case of the perfect matching polytope. Because it is based on the rectangle covering bound, our argument does not yield a super-polynomial lower bound on the extension complexity of the perfect matching polytope. This question was recently answered in the affirmativeby Rothvo\ss \cite{R13} posted on arXiv a proof of the fact that the extension complexity of the perfect matching polytope is $2^{\Omega(n)}$. This groundbreaking result is based on a general lower bound called the \emph{hyperplane separation bound}, which was used implicitly, e.g., in~\cite{bfps2012}.

As mentioned at the end of the introduction, the new connections developed have already inspired much follow-up research in particular about \emph{approximate} EFs. Here are two concrete questions left open for future work: (i) find a \emph{slack matrix} that has an exponential gap between nonnegative rank and PSD rank; (ii) prove that the cut polytope has no polynomial-size \emph{semidefinite}~EF  (that would rule out SDP-based algorithms for optimizing over the cut polytope, in the same way that this paper ruled out LP-based algorithms).


Our final remark concerns the famous \emph{log-rank conjecture}~\cite{LovaszSaks93}. It states that the deterministic communication complexity of a (finite) Boolean matrix $M$ is upper bounded by a polynomial in the logarithm of its rank~$\rk(M)$.  On the one hand, this conjecture is equivalent to the following statement: $\log(\nnegrk(M)) \leqslant \mathrm{polylog}(\rk(M))$ for all Boolean matrices~$M$. On the other hand, we know that $\psdrk(M) = O(\rk(M))$ for all Boolean matrices $M$ by Theorem~\ref{th:qupperlowrank}. Using the interpretation of the nonnegative and PSD rank of~$M$ in terms of classical and quantum communication protocols computing $M$ in expectation (see \cite{FaenzaFioriniGrappeTiwary11} and Theorem~\ref{thm:qProtImpliesFact}), we see that the log-rank conjecture is \emph{equivalent} to the conjecture that classical protocols computing~$M$ in expectation are at most polynomially less efficient than quantum protocols.  Accordingly, one way to prove the log-rank conjecture would be to 
give an efficient classical simulation of such quantum protocols for Boolean~$M$ (for \emph{non-Boolean}~$M$, we already exhibited an exponential separation in this paper).

\subsection*{Acknowledgments}

We thank Kota Ishihara for carefully reading the manuscript and pointing out an error in a previous version of the text. We thank Monique Laurent for information about hypermetric inequalities, and the three anonymous STOC'12 referees as well as one JACM referee for suggesting improvements to the text. Sebastian Pokutta would like to thank Alexander Martin for the inspiring discussions and support. Ronald de Wolf thanks Giannicola Scarpa and Troy Lee for useful discussions.

Samuel Fiorini acknowledges support from the \emph{Actions de Recherche Concert\'ees} (ARC) fund of the French community of Belgium. Serge Massar acknowledges support from the European Commission under the projects QCS (Grant No.\ 255961) and QALGO (Grant No.\ 600700). Hans Raj Tiwary was postdoctoral researcher of the \emph{Fonds National de la Recherche Scientifique} (F.R.S.--FNRS). Ronald de~Wolf was partially supported by a Vidi grant from the Netherlands Organization for Scientific Research (NWO), by ERC Consolidator grant QPROGRESS, and by the European Commission under the projects QCS (Grant No.\ 255961) and QALGO (Grant No.\ 600700).

\bibliographystyle{plain}
\bibliography{bibliography}

\appendix

\section{Background on Polytopes} \label{apx:background}

A \emph{(convex) polytope} is a set $P \subseteq \RR^d$ that is the convex hull $\conv(V)$ of a finite set of points $V$. Equivalently, $P$ is a polytope if and only if $P$ is bounded and the intersection of a finite collection of closed halfspaces. This is equivalent to saying that $P$ is bounded and the set of solutions of a finite system of linear inequalities and possibly equalities (each of which can be represented by a pair of inequalities).

Let $P \subseteq \RR^d$ be a polytope. A closed halfspace $H^+$ that contains $P$ is said to be \emph{valid} for~$P$. In this case the hyperplane $H$ that bounds $H^+$  is also said to be \emph{valid} for $P$. A \emph{face} of~$P$ is either $P$ itself or the intersection of $P$ with a valid hyperplane. Every face of a polytope is again a polytope. A face is called \emph{proper} if it is not the polytope itself. A \emph{vertex} is a minimal nonempty face. A \emph{facet} is a maximal proper face. An inequality $c^{\intercal} x \leqslant \delta$ is said to be \emph{valid} for $P$ if it is satisfied by all points of $P$. The face it defines is $F := \{x \in P \mid c^{\intercal} x = \delta\}$. The inequality is called \emph{facet-defining} if $F$ is a facet. The \emph{dimension} of a polytope \(P\) is the dimension of its affine hull \(\text{aff}(P)\).

Every (finite or infinite) set $V$ such that $P = \conv(V)$ contains all the vertices of $P$. Conversely, letting $\vertexset(P)$ denote the vertex set of $P$, we have $P = \conv(\vertexset(P))$. Suppose now that $P$ is \emph{full-dimensional}, i.e., $\dim(P) = d$. Then, every (finite) system $Ax \leqslant b$ such that $P = \{x \in \RR^d \mid Ax \leqslant b\}$ contains all the facet-defining inequalities of $P$, up to scaling by positive numbers. Conversely, $P$ is described by its facet-defining inequalities. 

If $P$ is not full-dimensional, these statements have to be adapted as follows. Every (finite) system describing $P$ contains all the facet-defining inequalities of $P$, up to scaling by positive numbers and adding an inequality that is satisfied with equality by \emph{all} points of $P$. Conversely, a linear description of $P$ can be obtained by picking one inequality per facet and adding a system of equalities describing \(\text{aff}(P)\).

A \emph{$0/1$-polytope} in $\RR^d$ is simply the convex hull of a subset of $\{0,1\}^d$.

A (convex) \emph{polyhedron} is a set $P \subseteq \RR^d$ that is the intersection of a finite collection of closed halfspaces. A polyhedron $P$ is a polytope if and only if it is bounded.

For more background on polytopes and polyhedra, see the standard reference~\cite{Ziegler}.

\end{document}